\documentclass{amsart}
\usepackage{amsmath}
\usepackage{amssymb}
\usepackage[all]{xy}
\usepackage{enumerate}
\numberwithin{equation}{section}

\theoremstyle{plain}
\newtheorem{thm}{Theorem}[section]
\newtheorem{lemma}[thm]{Lemma}

\newtheorem{prop}[thm]{Proposition}

\theoremstyle{definition}

\newtheorem{defn}[thm]{Definition}

\newtheorem{remark}[thm]{Remark}

\begin{document}

\title{On Exponentiation and Infinitesimal One-Parameter Subgroups of Reductive Groups}

\author[Paul Sobaje]
{Paul Sobaje}

\begin{abstract}

Let $G$ be a reductive algebraic group over an algebraically closed field $k$ of characteristic $p>0$, and assume $p$ is good for $G$.  Let $P$ be a parabolic subgroup with unipotent radical $U$.  For $r \ge 1$, denote by $\mathbb{G}_{a(r)}$ the $r$-th Frobenius kernel of $\mathbb{G}_a$.  We prove that if the nilpotence class of $U$ is less than $p$, then any embedding of $\mathbb{G}_{a(r)}$ in $U$ lies inside a one-parameter subgroup of $U$, and there is a canonical way in which to choose such a subgroup.  Applying this result, we prove that if $p$ is at least as big as the Coxeter number of $G$, then the cohomological variety of $G_{(r)}$ is homeomorphic to the variety of $r$-tuples of commuting elements in $\mathcal{N}_1(\mathfrak{g})$, the $[p]$-nilpotent cone of $Lie(G)$.
\end{abstract}

\maketitle

\section{Introduction and Statement of Results}

The motivation behind the results in this paper comes from a few different sources.  To begin, let $G$ be a reductive algebraic group over an algebraically closed field $k$ of characteristic $p>0$, and assume further that the prime $p$ is \textit{good} for the root system of $G$.  Let $P$ be a parabolic subgroup of $G$ having unipotent radical $U$, and denote by $\mathfrak{g},\mathfrak{p}$ and $\mathfrak{u}$ the Lie algebras of these various groups.  Such Lie algebras are equipped with a restriction map $x \mapsto x^{[p]}$, and we denote by $\mathcal{N}(\mathfrak{g})$ (resp. $\mathcal{N}_1(\mathfrak{g})$) the nullcone (resp. $[p]$-nilpotent cone) of $\mathfrak{g}$.

In \cite[\S 2]{Ser} J.-P. Serre gave an argument, the extent of which was more fully presented by G. Seitz in \cite[\S 5]{Sei}, establishing that if the nilpotence class of $U$ is less than $p$, then there exists an exponential map, $\text{exp}: \mathfrak{u} \rightarrow U$, constructed by taking the exponential map for the corresponding Lie algebra and algebraic group over $\mathbb{Q}$, proving that the map is defined over $\mathbb{Z}_{(p)}$, and then base-changing to $k$.  It is an isomorphism of algebraic groups, the group structure on $\mathfrak{u}$ coming from the Hausdorff formula, and was further proved to be the unique isomorphism which has tangent map equal to the identity on $\mathfrak{u}$, and which is also compatible with the conjugation action of $P$ \cite[Proposition 5.3]{Sei}.  J. Carlson, Z. Lin, and D. Nakano further showed in \cite{CLN} that if $G$ is simple, and if $\mathcal{N}_1(\mathfrak{g})$ is a normal variety, then this exponential map can be extended to a $G$-equivariant isomorphism between $\mathcal{N}_1(\mathfrak{g})$ and $\mathcal{U}_1(G)$, the $p$-unipotent subvariety of $G(k)$.  In particular if $p \ge h$, the Coxeter number of $G$, then this extension exists (a fact independently observed by G. McNinch in \cite[Remark 39]{M2}).

We denote by $\mathbb{G}_{a(r)}$ the $r$-th Frobenius kernel of the additive group $\mathbb{G}_a$.  An \textit{infinitesimal one-parameter subgroup} of $G$ is a homomorphism of affine group schemes over $k$, $\varphi: \mathbb{G}_{a(r)} \rightarrow G$, for some $r \ge 1$.  Following \cite{SFB1} we will refer to the set of morphisms $\textup{Hom}_{gs/k}(\mathbb{G}_{a(r)},G)$ as the variety of infinitesimal one-parameter subgroups of $G$ of height $\le r$.  That this set has the structure of an affine variety is established in \cite[Theorem 1.5]{SFB1}.  In the case $r=1$, one can show that there is a natural isomorphism

\begin{equation}\label{iso}
\textup{Hom}_{gs/k}(\mathbb{G}_{a(1)},G) \cong \mathcal{N}_1(\mathfrak{g}).
\end{equation}
\vspace{0.02in}

\noindent Suppose now that $U$ is as above.  If $x \in \mathfrak{u} = \mathcal{N}_1(\mathfrak{u})$ (\cite[4.4]{M}), let $\varphi_x$ be the corresponding element in $\textup{Hom}_{gs/k}(\mathbb{G}_{a(1)},U)$ given by (\ref{iso}).  The exponential map constructed in \cite[Proposition 5.3]{Sei} produces a one-parameter subgroup 

\vspace{0.03in}
\begin{center}$\text{exp}_x: \mathbb{G}_a \rightarrow G, \quad s \mapsto \text{exp}(sx)$.\end{center}
\vspace{0.03in}

\noindent This one-parameter subgroup, when restricted to $\mathbb{G}_{a(1)} \le \mathbb{G}_{a}$, will equal the morphism $\varphi_x$.  Thus, the existence of the exponential map on $\mathfrak{u}$ can be used to prove that every height 1 infinitesimal one-parameter subgroup of $U$ lies inside a one-parameter subgroup of $U$, while the uniqueness of the exponential provides a canonical, or preferred, such subgroup for each $x \in \mathfrak{u}$.  Stated another way, exponentiation on $\mathfrak{u}$ gives a canonically defined map of sets 

\vspace{0.03in}
\begin{center}$\text{exp}_U^1: \textup{Hom}_{gs/k}(\mathbb{G}_{a(1)},U) \rightarrow \textup{Hom}_{gs/k}(\mathbb{G}_a,U)$\end{center}
\vspace{0.03in}

\noindent with certain desirable properties.

\bigskip
In this paper, our first objective is to generalize this notion to arbitrary Frobenius kernels of $\mathbb{G}_a$.  Specifically, we prove:

\begin{thm}\label{first}
Let $G$ be a reductive group, and assume that $p$ is good for $G$.  Let $P$ be a parabolic subgroup of $G$ with unipotent radical $U$, and assume that the nilpotence class of $U$ is less than $p$.  Then, for each $r \ge 1$, there is a canonically defined map

\vspace{0.05in}
\begin{center}$\textup{exp}^r_U: \textup{Hom}_{gs/k}(\mathbb{G}_{a(r)},U) \rightarrow \textup{Hom}_{gs/k}(\mathbb{G}_a,U)$,\end{center}
\vspace{0.05in}

\noindent such that for all $\varphi, \phi \in \textup{Hom}_{gs/k}(\mathbb{G}_{a(r)},U)$:

\begin{enumerate}

\item $\textup{exp}^r_U(\varphi){\mid}_{\mathbb{G}_{a(r)}} = \varphi$.

\item If $\sigma$ is an automorphism of $P$, then $\textup{exp}^r_U(\sigma \, \circ \, \varphi) = \sigma \circ \textup{exp}^r_U(\varphi)$.  In particular, $\textup{exp}^r_U$ is compatible with the conjugation action of $P$.

\item $[\varphi(a), \phi(b)] = 1$ for all $a,b \in \mathbb{G}_{a(r)}(A)$ and all commutative $k$-algebras $A$, if and only if $[\textup{exp}^r_U(\varphi)(c), \textup{exp}^r_U(\phi)(d)] = 1$ for all $c,d \in \mathbb{G}_a(A)$, and all $A$.

\end{enumerate}

\end{thm}

The proof of this theorem does not involve reduction mod $p$, but rather analyzes Hopf algebra homomorphisms between the coordinate algebras of the relevant affine group schemes.  Nonetheless, we show that when $r=1$ this result agrees with the extension of height 1 infinitesimal one-parameter subgroups coming from the work of \cite[5.3]{Sei}.

We next show that the method of proof for Theorem \ref{first} allows for a fairly easy extension to a $G$-orbit.  Recall that the conjugation action of $G$ on itself induces an action of $G$ on $\textup{Hom}_{gs/k}(\mathbb{G}_{a(r)},G)$, of which $\textup{Hom}_{gs/k}(\mathbb{G}_{a(r)},U)$ is a subset.

\begin{thm}\label{second}
Let $p$ be good for $G$, and let $P$ be a parabolic subgroup of $G$ with unipotent radical $U$ having nilpotence class less than $p$.  For each $r \ge 1$, the map $\textup{exp}^r_U$ extends to a map

\vspace{0.05in}
\begin{center}$\textup{exp}^r: G \cdot \textup{Hom}_{gs/k}(\mathbb{G}_{a(r)},U) \rightarrow \textup{Hom}_{gs/k}(\mathbb{G}_a,G)$,\end{center}
\vspace{0.05in}

\noindent given by $\textup{exp}^r(g \cdot \varphi) := g \cdot \textup{exp}^r_U(\varphi)$, \, for all $g \in G(k), \varphi \in \textup{Hom}_{gs/k}(\mathbb{G}_{a(r)},U)$.
\end{thm}

\bigskip
The main application we give for these results is in describing varieties of infinitesimal one-parameter subgroups.  Let $H_{(r)}$ be the $r$-th Frobenius kernel of an affine algebraic group $H$.  The work of E. Friedlander and A. Suslin in \cite{FS} paved the way for adapting cohomological techniques, used in the representation theory of finite groups and restricted Lie algebras, to study representations of $H_{(r)}$.  The authors, in collaboration with C. Bendel, then followed up on this in \cite{SFB1} and \cite{SFB2} where they proved in particular that the maximal ideal spectrum of the cohomology ring of $H_{(r)}$ is homeomorphic to the variety of infinitesimal one-parameter subgroups  of $H$ of height $\le r$.  As noted in \cite{SFB1}, this result provides an incentive for finding an explicit description of $\textup{Hom}_{gs/k}(\mathbb{G}_{a(r)},H)$.  We recalled in (\ref{iso}) that there is such a description in the case $r=1$ (this result holding for any algebraic group $H$, not only reductive groups).  However, when $r>1$, the situation becomes much more complicated.

It was shown in \cite{SFB1} that for $GL_n$ there is a natural identification between $\textup{Hom}_{gs/k}(\mathbb{G}_{a(r)},GL_n)$ and the variety of $r$-tuples of commuting elements in $\mathcal{N}_1(\mathfrak{gl}_n)$.  Unfortunately this description does not immediately transfer to arbitrary closed subgroups of $GL_n$ (replacing $\mathcal{N}_1(\mathfrak{gl}_n)$ with $\mathcal{N}_1(\mathfrak{g})$), but only those subgroups which are of \textit{exponential type}.   It turns out that various classical subgroups of $GL_n$ satisfy this property (see Proposition 1.2 and Lemmas 1.7,\, 1.8 of \cite{SFB1}), but in general determining which subgroups are of this type is somewhat difficult.  In \cite[Theorem 9.5]{M}, G. McNinch was able to prove that every semisimple group $G$ has an embedding into some $GL_n$ which is of exponential type, assuming that $p$ is large enough, the bound depending on $G$.  In particular, the methods of \cite{M} show that if $G$ is an adjoint semisimple group, the condition $p > 2h-2$ is sufficient.

In our final result, we avoid exponential type embeddings altogether by proving directly that if $G$ is reductive and $p \ge h$, then $\textup{Hom}_{gs/k}(\mathbb{G}_{a(r)},G)$ has a description analogous to that for $GL_n$.  Noting in this case that $\mathcal{N}(\mathfrak{g}) = \mathcal{N}_1(\mathfrak{g})$, we prove: 

\begin{thm}\label{fourth}
Let $G$ be a reductive group.  Assume that $p \ge h$, and that $PSL_p$ is not a simple factor of $[G,G]$.  Then there is a natural identification between $\textup{Hom}_{gs/k}(\mathbb{G}_{a(r)},G)$ and the variety

\vspace{0.06in}
\begin{center}$C_r(\mathcal{N}(\mathfrak{g})): = \{(x_0,x_1,\ldots,x_{r-1}) \mid x_i \in \mathcal{N}(\mathfrak{g}), [x_i,x_j] =0\}$.\end{center}
\vspace{0.06in}

\noindent This identification sends the $r$-tuple $(x_0,x_1,\ldots,x_{r-1})$ to the morphism

\vspace{0.06in}
\begin{center}$\textup{exp}_{x_0}\textup{exp}_{x_1}^{(1)} \cdots \textup{exp}_{x_{r-1}}^{(r-1)}$,\end{center}
\vspace{0.06in}

\noindent where for $a \in \mathbb{G}_a(k)$, $\textup{exp}_{x_i}^{(i)}(a) = \textup{exp}_{x_i}(a^{p^i})$. 

\end{thm}

This result is of interest both because of what it proves, and also because its proof is intrinsic, not depending on an embedding of $G$ into $GL_n$.  This directly addresses a question raised in \cite[Remark 1.9]{SFB1}, and builds upon some earlier work in this direction by McNinch (\cite{M}, introduction and section 9.6).

\bigskip
This paper is organized as follows.  Section 2 begins by recalling some of the relevant concepts and terminology, and then proceeds in proving lemmas and propositions which culminate in establishing Theorem \ref{first}.  Section 3 is then devoted to extending this to $G$-orbits, proving Theorem \ref{second}.  The final section deals infinitesimal one-parameter subgroups and a proof of Theorem \ref{fourth}. 

\section{Exponentiation for Parabolic Subgroups}

Let $G$ be a reductive algebraic group over $k$ with parabolic subgroup $P$.  We will assume throughout that $p$ is a good prime for $G$.  This implies that $p>2$ if $G$ has a component of type $B,C$ or $D$; $p>3$ if $G$ has a component of type $E_6,E_7,F_4$ or $G_2$; and $p>5$ if $E_8$ is a component of $G$.

Choose $B$ to be a Borel subgroup of $P$, and $T$ a maximal torus contained in $B$.  We will denote by $U$ the unipotent radical of $P$.  As we will need to view each of these objects within the larger category of affine group schemes over $k$, we will henceforth use $G(k), P(k)$, etc. to refer specifically to the $k$-rational points of the given scheme.  We recall that an affine group scheme $H$ over $k$ is a representable functor from the category of commutative $k$-algebras to groups, and we denote by $k[H]$ the coordinate (or representing) algebra of $H$.  It is a commutative Hopf algebra over $k$, and in the case that $H$ is the affine group scheme coming from an algebraic group, then $k[H]$ is also the ring of regular functions on $H(k)$.  Any morphism of affine group schemes $f: H_1 \rightarrow H_2$ is equivalent to a comorphism $f^{*}: k[H_2] \rightarrow k[H_1]$ which is a homomorphism of Hopf algebras (cf. \cite{W}).

The group scheme $\mathbb{G}_a$ has coordinate algebra $k[t]$, with the comultiplication specified by $\Delta(t) = t \otimes 1 + 1 \otimes t$.  The coordinate algebra of $\mathbb{G}_{a(r)}$ is the quotient $k[t]/(t^{p^r})$, $(t^{p^r})$ being a Hopf ideal of $k[t]$.  At this point, let us outline the basic idea behind our approach.  Any morphism $\varphi: \mathbb{G}_{a(r)} \rightarrow U$ is given by a comorphism $\varphi^{*}: k[U] \rightarrow k[t]/(t^{p^r})$.  As we will see in a moment, $k[U]$ is a free $k$-algebra (in fact, this is true for any connected unipotent algebraic group over $k$).  It follows that we can ``lift" this comorphism to an algebra map $\widetilde{\varphi^{*}}: k[U] \rightarrow k[t]$ such that $\varphi^{*} = \pi_r \circ \widetilde{\varphi^{*}}$, where $\pi_r$ is the canonical projection $k[t] \rightarrow k[t]/(t^{p^r})$.   To obtain a one-parameter subgroup of $U$, we must first define exactly what lifting of $\varphi^{*}$ we are taking, and then prove that this map is a morphism of Hopf algebras.  But before doing this, we will need to establish some facts about the algebra and coalgebra structure of $k[U]$.

Let $\Phi$ be the root system of $G$ with respect to $T$.  For each $\alpha \in \Phi$, fix a root homomorphism $u_{\alpha}: \mathbb{G}_a \rightarrow G$.  Let $U_{\alpha}$ be the image of $\mathbb{G}_a$ under $u_{\alpha}$.  We will choose our set of simple roots $\Pi \subseteq \Phi$ so that $B$ is generated by $T$ and all root subgroups $U_{\alpha}, \alpha \in \Phi^+$.  There is a $J \subseteq \Pi$ such that $P$ is generated by $B$ and the root subgroups $U_{\alpha}$, $\alpha \in \Phi_J$, where $\Phi_J = \mathbb{Z}J \cap \Phi$ \cite[8.4.3]{Sp}.  By setting $\Phi_J^+ = \Phi_J \cap \Phi^+$, and $\Phi_J^- = - \Phi_J^+$, we note that $P$ is generated by $B$ and the root subgroups corresponding to $\Phi_J^-$.  The unipotent radical $U$ is the subgroup scheme of $P$ generated by the root subgroups $U_{\alpha}$, $\alpha \in \Phi^+ \backslash \Phi_J^+$.

Fix a total order on the set of roots, and let $n = |\Phi^+ \backslash \Phi_J^+|$.  As can be deduced from \cite[8.2.1]{Sp}, there is an isomorphism of varieties

\vspace{0.08in}
\begin{center}$\mathbb{G}_a^n(k) \stackrel{\sim}{\longrightarrow} U(k)$,\end{center}
\vspace{0.08in}

\noindent given by sending $(a_1, a_2, \ldots, a_n)$ to $u_{\alpha_1}(a_1)u_{\alpha_2}(a_2)\cdots u_{\alpha_n}(a_n)$, the order on the right given by our total order.  We can then identify $k[U]$ as a polynomial ring $k[Y_{\alpha_1}, Y_{\alpha_2}, \ldots, Y_{\alpha_n}]$, where

\vspace{0.08in}
\begin{center} $Y_{\alpha_i}\left(u_{\alpha_1}(a_1)u_{\alpha_2}(a_2)\cdots u_{\alpha_n}(a_n)\right) = a_i$. \end{center}
\vspace{0.08in}

If $\beta \in \Phi$, it can be uniquely expressed as $\beta = \sum n_{\alpha}\alpha, \; \alpha \in \Pi$.  Following D. Testerman \cite{Test}, we define the function $ht_J: \Phi \rightarrow \mathbb{Z}$ given by $ht_J(\beta) = \sum_{\alpha \in \Pi \backslash J} n_{\alpha}$.  This function extends linearly to $\mathbb{Z}\Phi$.  We recall that the nilpotence class of $U$ is defined to be the least integer $e$ such that $U(k)^{(e)} = \{1\}$, where $U(k)^{(0)} := U(k)$, and  $U(k)^{(i)} := [U(k),U(k)^{(i-1)}]$, where here the bracket denotes the commutator of the two subgroups.  As noted in the paragraph preceding \cite[Proposition 3.5]{Sei}, the nilpotence class of $U$ is equal to the maximum value of $ht_J(\alpha)$, as $\alpha$ ranges over the positive roots (the statement in \cite{Sei} deals with simple groups, but adapts easily to reductive groups following the observation in \cite[Lemma 4.3]{M}).

Let $M$ be a $T$-module, and denote by $X(T)$ the character group of $T$.  We will write $M_{\lambda}$ for the subspace of $M$ on which $T$ acts by $\lambda \in X(T)$.  The right action of $P$ on $U$ by conjugation gives $k[U]$ the structure of a left $P$-module.  In particular, the action of $T$ on $U$ makes $k[U]$ into a $T$-module, the element $Y_{\alpha}$ having weight $-\alpha$ \cite[II.1.7]{J1}.  It is easy to see then that the monomial $\prod Y_{\alpha_i}^{n_{\alpha_i}}$ has weight $\sum n_{\alpha_i}(-\alpha_i)$, so that the weights of $k[U]$ are precisely those $\lambda \in \mathbb{Z}_{\le 0}(\Phi^+ \backslash \Phi^+_J)$.  In particular, each such $\lambda$ has the property that $ht_J(\lambda) \le 0$.

It will also be important for us to observe that by letting $P$ act on $U \times U$ componentwise, the multiplication map $U \times U \rightarrow U$ is $P$-equivariant, so that the comultiplication $\Delta: k[U] \rightarrow k[U] \otimes k[U]$ is a homomorphism of $P$-modules.  Thus 

\begin{equation}\label{equivariant}
\Delta(k[U]_{\lambda}) \subseteq (k[U] \otimes k[U] )_{\lambda} = \sum_{\lambda_1 + \lambda_2 = \lambda} k[U]_{\lambda_1} \otimes k[U]_{\lambda_2}
\end{equation}

\bigskip
The next lemma and definition will be useful in stating precisely what we mean by lifting the comorphism $\varphi^{*}: k[U] \rightarrow k[t]/(t^{p^r})$ to an algebra map from $k[U]$ to $k[t]$.

\bigskip
\begin{lemma}\label{Vspace}
Let $V$ be the subspace of $k[t]$ consisting of all polynomials of degree less than $p^r$.  Then the projection $\pi_r: k[t] \rightarrow k[t]/(t^{p^r})$ induces an isomorphism of coalgebras $V \cong k[t]/(t^{p^r})$.
\end{lemma}

\begin{remark}
Although this result is clear, we mention it primarily to point out that the comultiplication on the two spaces is identical.
\end{remark}

\begin{defn}\label{obvious}
Let $X := \{x_1,\ldots,x_n\}$ be a subset of $k[U]$ on which $k[U]$ is free as a $k$-algebra.  Let $f$ be a homomorphism of $k$-algebras from $k[U]$ to $k[t]/(t^{p^r})$.  We define ${f_X}: k[U] \rightarrow k[t]$ to be the $k$-algebra map defined by sending each $x_i$ to the unique polynomial of degree less than $p^r$ in $\pi_r^{-1}(f(x_i))$.  In other words, $f_X(x_i) = V \cap \pi_r^{-1}(f(x_i))$.
\end{defn}

\bigskip
Now, let $\varphi: \mathbb{G}_{a(r)} \rightarrow U$ be a morphism of affine group schemes with comorphism $\varphi^{*}: k[U] \rightarrow k[t]/(t^{p^r})$, and set $Y := \{Y_{\alpha_1}, \ldots, Y_{\alpha_n}\}$.  We will show that when the nilpotence class of $U$ is less than $p$, so that $ht_J(\alpha) < p$ for all $\alpha \in \Phi^+$, the lifting $\varphi^{*}_Y: k[U] \rightarrow k[t]$ is a morphism of Hopf algebras.  To simplify later proofs, we establish some useful lemmas.  

\begin{lemma}\label{degree}
Let $x \in k[U]_{\lambda}$, $\lambda \ne 0$, and write

\vspace{0.1in}
\begin{center}$\varphi^{*}(x) = a_1t + a_2t^2 + \cdots + a_{p^r-1}t^{p^r-1}, \; a_i \in k$.\end{center}
\vspace{0.1in}

If $\varphi^{*}(x) \ne 0$, let $m$ be the largest integer such that $a_m \ne 0$.  Then $m \le |ht_J(\lambda)|p^{r-1}$.

\end{lemma}

\begin{proof}
We will proceed by induction on $|ht_J(\lambda)|$.  Note that if $|ht_J(\lambda)| = 1$, then $\lambda = -\alpha$ for some $\alpha \in \Phi^+ \backslash \Phi_J^+$, for if $\lambda$ was equal to the negative sum of two roots in $\Phi^+ \backslash \Phi_J^+$, then the additivity of $ht_J$ would imply that $|ht_J(\lambda)| \ge 2$.  As $x$ is in $I$, the augmentation ideal of $k[U]$, by \cite[Ex. 2.3]{W} we know that $\Delta(x) = x \otimes 1 + 1 \otimes x + z, z \in I \otimes I$.  However, as $|ht_J(\lambda)| = 1$, the observation about the comultiplication of $k[U]$ given in (\ref{equivariant}) forces $z=0$, so that $\Delta(x) = x \otimes 1 + 1 \otimes x$.  This implies that $a_i \ne 0$ only if $i \in \{1, p, p^2, \ldots, p^{r-1}\}$.  Thus $m \le p^{r-1}$, so that the claim is true for $|ht_J(\lambda)| = 1$.

Now suppose that $|ht_J(\lambda)| = j>1$, and that the claim is true for all $x^{\prime} \in k[U]_{\lambda^{\prime}}$ where $|ht_J(\lambda^{\prime})| < j$.  If $m \le p^{r-1}$ then we are done, so assume $m > p^{r-1}$.  A basis for $k[t]/(t^{p^r}) \otimes k[t]/(t^{p^r})$ is given by all simple tensors of the form $t^{c_1} \otimes t^{c_2}$, where $0 \le c_i < p^r$.  Denote this basis by $\mathcal{B}$.  We see then that $\Delta(\varphi^{*}(x))$, written in terms of the basis $\mathcal{B}$, must contain a term of the form $bt^{c_1} \otimes t^{c_2}$, where $b,c_1,c_2 \ne 0$ and $c_1 + c_2 = m$.  On the other hand, by (\ref{equivariant}), it follows that

$$\Delta(x) \in x \otimes 1 + 1 \otimes x + \sum_i y_i \otimes z_i$$

\noindent where

$$y_i \otimes z_i \in \sum_{\lambda_1 + \lambda_2 = \lambda, \lambda_i \ne 0} k[U]_{\lambda_1} \otimes k[U]_{\lambda_2}$$

\bigskip
As both $c_1,c_2 \ne 0$, we see that $bt^{c_1} \otimes t^{c_2}$ must lie in the space

$$\sum_{\lambda_1 + \lambda_2 = \lambda, \lambda_i \ne 0} \varphi^{*}(k[U]_{\lambda_1}) \otimes \varphi^{*}(k[U]_{\lambda_2})$$

\bigskip
The condition underneath the summation symbol implies that $|ht_J(\lambda_i)| < |ht_J(\lambda)|$, thus allowing the induction hypothesis to be applied to all elements in $k[U]_{\lambda_i}$.  From this we conclude that
$$m=c_1+c_2 = |ht_J(\lambda_1)|p^{r-1} + |ht_J(\lambda_2)|p^{r-1} \le jp^{r-1} = |ht_J(\lambda)|p^{r-1}.$$
\end{proof}

\bigskip

In general, the lift in Definition \ref{obvious} will be dependent on the set of algebra generators which we choose.  However, the next lemma will show that this is not a concern in the special case that we are considering.  Let $k[U]_{<p}$ be the subspace of $k[U]$ defined by

\begin{equation}\label{S}
k[U]_{<p} := \sum_{|ht_J(\lambda)| < p} k[U]_{\lambda}.
\end{equation}

If the nilpotence class of $U$ is less than $p$, we are guaranteed that a set of algebra generators lies in $k[U]_{<p}$ (namely, $Y \subseteq k[U]_{<p}$).  The following lemma shows that any two choices of algebra generators from $k[U]_{<p}$ will yield the same lifting.

\begin{lemma}\label{changes}
Assume that $p$ is greater than the nilpotence class of $U$, and let $k[U]_{<p}$ be the subspace of $k[U]$ specified in (\ref{S}).  Let $\varphi^{*}: k[U] \rightarrow k[t]/(t^{p^r})$ be the comorphism of $\varphi$ as above.  Then the lifting $\varphi^{*}_Y:k[U] \rightarrow k[t]$ has the property that for each $x \in k[U]_{<p}$, 

\vspace{0.05in}
\begin{center} $\varphi^{*}_Y(x) = V \cap \pi_r^{-1}(\varphi^{*}(x)).$  \end{center}
\vspace{0.05in}

\noindent In particular, if $X = \{x_1, \ldots, x_n\}$ is another subset of $k[U]_{<p}$ such that $k[U]$ is a free $k$-algebra on $X$, then $\varphi^{*}_X = \varphi^{*}_Y$.
\end{lemma}

\begin{proof}
An implication of Lemma \ref{degree} is that for any weight space $k[U]_{\lambda}$, and any $x \in k[U]_{\lambda}$, the degree of $\varphi^{*}_Y(x)$ as a polynomial in $k[t]$ is bounded by $|ht_J(\lambda)|p^{r-1}$.  To see this, we note that it is true for the elements $Y_{\alpha}$, and then by extension to all monomials in $k[U]$, the monomials providing basis elements for the various weight spaces.

By this observation, it follows that $\varphi^{*}_Y(k[U]_{<p}) \subseteq V$.  This proves the first claim, and the second follows from the first.
\end{proof}

\bigskip

\begin{prop}\label{HopfGood}
If the nilpotence class of $U$ is less than $p$, then $\varphi^{*}_Y$ is a homomorphism of Hopf algebras.  The definition of $\varphi^{*}_Y$ is also independent both of the total ordering of the positive roots, and of the choice of root homomorphisms $u_{\alpha}$.
\end{prop}

\begin{proof}
To prove the first statement, it suffices that we check that the comultiplicative structures on each Hopf algebra are preserved.  But this is the same as showing that two different algebra homomorphisms from $k[U]$ to $k[t] \otimes k[t]$ agree on all elements in $k[U]$, thus we can further reduce to proving that comultiplication is preserved on the set of algebra generators given by $Y$.

As noted in Lemma \ref{Vspace}, any element $x \in V$ has the ``same" comultiplication as does $\pi_r(x) \in k[t]/(t^{p^r})$.  We will thus have that $\Delta(\varphi^{*}_Y(Y_{\alpha})) = \varphi^{*}_Y \otimes \varphi^{*}_Y \circ \Delta (Y_{\alpha})$ if and only if $\varphi^{*}_Y \otimes \varphi^{*}_Y \circ \Delta (Y_{\alpha}) = \sum R_i(t) \otimes Q_i(t)$ where all $R_i(t), Q_i(t) \in V$.  Since we are assuming that $p > |ht_J(\alpha)|$ for all $\alpha \in \Phi^+ \backslash \Phi^+_J$, it follows from Lemma \ref{degree} and (\ref{equivariant}) that the degree of each polynomial $R_i(t), Q_i(t)$ is less than $p\cdot p^{r-1} = p^r$.

To prove the second claim, we see that a different choice of root homomorphisms and/or total order on the roots amounts to a different choice of algebra generators $Y_{\alpha}^{\prime}$.  We note however that $Y_{\alpha}^{\prime}$ also has weight $-\alpha$ (though is not necessarily a scalar multiple of $Y_{\alpha}$).  If we let $Y^{\prime}$ denote our new choice of algebra generators, the previous remarks show that $Y^{\prime} \subseteq k[U]_{<p}$.  We now can apply Lemma \ref{changes} which states that $\varphi^{*}_Y = \varphi^{*}_{Y^{\prime}}$.
\end{proof}

\bigskip
We can now define a set map 

\vspace{0.05in}
\begin{center}$\textup{exp}^r_U: \textup{Hom}_{gs/k}(\mathbb{G}_{a(r)},U) \rightarrow \textup{Hom}_{gs/k}(\mathbb{G}_a,U)$,\end{center}
\vspace{0.05in}

\noindent which takes $\varphi \in \textup{Hom}_{gs/k}(\mathbb{G}_{a(r)},U)$ to the element in $\textup{Hom}_{gs/k}(\mathbb{G}_a,U)$ having comorphism $\varphi^{*}_Y$.  Theorem \ref{exp} will show that this map satisfies certain important properties, and that it is canonically defined.  We must, however, make a few more recollections and definitions before stating the theorem.

\bigskip
First, we recall that the inclusion $\mathbb{G}_{a(r)} \hookrightarrow \mathbb{G}_a$ induces a restriction map

\vspace{0.05in}
\begin{center}$res: \textup{Hom}_{gs/k}(\mathbb{G}_a,U) \rightarrow \textup{Hom}_{gs/k}(\mathbb{G}_{a(r)},U)$.\end{center}
\vspace{0.05in}

Next, we note that every automorphism of $P$ sends $U$ to itself, and thus acts (by composition of morphisms) on the sets $\textup{Hom}_{gs/k}(\mathbb{G}_a,U)$ and $\textup{Hom}_{gs/k}(\mathbb{G}_{a(r)},U)$.  If $h$ is such an automorphism, we say that $\text{exp}^r_U$ is compatible with $h$ provided that $\text{exp}^r_U(h \circ \varphi) = h \circ \text{exp}^r_U(\varphi)$, for all $\varphi \in \textup{Hom}_{gs/k}(\mathbb{G}_{a(r)},U)$.  We say that $\text{exp}^r_U$ is compatible with the conjugation action of $P(k)$ if it is compatible with all inner-automorphisms of $P$.

\bigskip
Finally, we will need to define the notion of commuting morphisms between affine group schemes.

\begin{defn}\label{commutes}
Let $H_1, H_2$ be affine group schemes over $k$, and let $\varphi, \varphi^{\prime}$ be elements in  $\textup{Hom}_{gs/k}(H_1,H_2)$.  Then we will say that $\varphi$ and $\varphi^{\prime}$ commute, and write $[\varphi, \varphi^{\prime}] = 0$, if the following composition

$$H_1 \times H_1 \xrightarrow{\varphi \times \varphi^{\prime}} H_2 \times H_2 \xrightarrow{m_{H_2}} H_2$$

\vspace{0.15in}
\noindent is a morphism of affine group schemes from $H_1 \times H_1$ to $H_2$, where $m_{H_2}$ denotes the multiplication map of the group scheme $H_2$.
\end{defn}

\bigskip

We can now state the main result of this section:

\begin{thm}\label{exp}
Let $P,B, T$, and $U$ be as above, and assume that $p$ is greater than the nilpotence class of $U$.  Then for every $r \ge 1$, there is a map of sets

\vspace{0.05in}
\begin{center}$\textup{exp}^r_U: \textup{Hom}_{gs/k}(\mathbb{G}_{a(r)},U) \rightarrow \textup{Hom}_{gs/k}(\mathbb{G}_a,U)$\end{center}
\vspace{0.05in}

\noindent where $\textup{exp}^r_U(\varphi)$ is the morphism with comorphism $\varphi^{*}_Y$.  This map satisfies the following properties:

\vspace{0.04in}
\begin{enumerate}

\item $res \circ \textup{exp}^r_U = id$.

\item If $\varphi, \phi \in \textup{Hom}_{gs/k}(\mathbb{G}_{a(r)},U)$, then

\vspace{0.07in}
$[\varphi, \phi] = 0 \quad \text{ if and only if } \quad [\textup{exp}^r_U(\varphi), \textup{exp}^r_U(\phi)] = 0$

\item $\textup{exp}^r_U$ is compatible with the conjugation action of $P$.

\item $\textup{exp}^r_U$ does not depend on the choice of $B$ or $T$, and is compatible with any automorphism of $P$.

\end{enumerate}
\vspace{0.04in}

\end{thm}

\begin{proof}
Proposition \ref{HopfGood} establishes the existence of $\textup{exp}^r_U$, and (1) is clear given the way we have defined the comorphism $\varphi^{*}_Y$.

(2) First we observe that $[\textup{exp}^r_U(\varphi), \textup{exp}^r_U(\varphi^{\prime})] = 0$ implies $[\varphi, \varphi^{\prime}] = 0$, which is easily seen by restricting the morphism in the first case to the subgroup scheme $\mathbb{G}_{a(r)} \times \mathbb{G}_{a(r)}$.  The main work involved in showing the reverse implication deals with reinterpretting our definition of ``commuting morphisms" in terms of their comorphisms, for then the proof will be reasonably straight-forward given our previous work.  The statement that

\vspace{0.05in}
\begin{center}$\mathbb{G}_a \times \mathbb{G}_a \xrightarrow{\textup{exp}^r_U(\varphi) \times \textup{exp}^r_U(\phi)} U \times U \xrightarrow{m_U} U$\end{center}
\vspace{0.05in}

\noindent is a morphism of affine group schemes is equivalent to saying that the composite of maps:

\vspace{0.05in}
\begin{center}$k[U] \xrightarrow{\Delta} k[U] \otimes k[U] \xrightarrow{\varphi^{*}_Y \otimes \phi^{*}_Y} k[t] \otimes k[t]$\end{center}
\vspace{0.05in}

\noindent is a morphism of Hopf algebras (and we have the analogous statement in the $\mathbb{G}_{a(r)}$ case).  Now this composite is a morphism of $k$-algebras (for both maps in the chain are), so we just must verify that it is a map of coalgebras.  Note that the coalgebra structure on $k[t] \otimes k[t]$ is given by 

\vspace{0.05in}
\begin{center}$k[t] \otimes k[t] \xrightarrow{\Delta \otimes \Delta} k[t] \otimes k[t] \otimes k[t] \otimes k[t] \xrightarrow{id. \otimes \tau \otimes id.} k[t] \otimes k[t] \otimes k[t] \otimes k[t]$\end{center}
\vspace{0.05in}

\noindent where $\tau$ is the twist map sending a simple tensor $a \otimes b$ to $b \otimes a$ (to see that this is the correct structure, one can write the diagram for the multiplicative structure placed on the tensor product of two $k$-algebras, and then take the ``dual" diagram).  Using arguments similar to those appearing earlier, we see that it suffices to check this map on the elements $Y_{\alpha}$ (because, as observed above, $(\varphi^{*}_Y \otimes \phi^{*}_Y) \circ \Delta$ is a $k$-algebra map).  It then follows that if

\vspace{0.05in}
\begin{center}$k[U] \xrightarrow{\Delta} k[U] \otimes k[U] \xrightarrow{\varphi^{*} \otimes \phi^{*}} k[t]/(t^{p^r}) \otimes k[t]/(t^{p^r})$\end{center}
\vspace{0.05in}

\noindent respects the comultiplication on the elements in $k[U]_{<p}$, then so will $(\varphi^{*}_Y \otimes \phi^{*}_Y) \circ \Delta$.  This shows that $[\varphi, \varphi^{\prime}] = 0$ implies $[\textup{exp}^r_U(\varphi), \textup{exp}^r_U(\varphi^{\prime})] = 0$.

(3) The group $P(k)$ is generated by $T(k)$ and the root subgroups $U_{\alpha}(k)$ such that $\alpha \in \Phi^+ \cup \Phi_J$.  By the definition of a weight space, the action of $T(k)$ on $k[U]$ stabilizes $k[U]_{\lambda}$ for each $\lambda$, thus it stabilizes $k[U]_{<p} \subseteq k[U]$.  On the other hand, the action of any $x \in U_{\alpha}(k)$ sends $k[U]_{\lambda}$ into $\sum k[U]_{\lambda + c\alpha}$  $(c \in \mathbb{Z}^+)$ \cite[27.2]{H}.  Since $ht_J(\alpha) \ge 0$ for all $\alpha \in \Phi^+ \cup \Phi_J$, and since $ht_J(\lambda) \le 0$, it follows that for any $c$ such that $k[U]_{\lambda + c\alpha} \ne \{0\}$ we must have $|ht_J(\lambda + c\alpha)| \le |ht_J(\lambda)|$.  Thus the action of $x$ also stabilizes $k[U]_{<p}$, hence all of $P(k)$ does.

Now let $x$ be any element in $P(k)$, let $x^{*}: k[U] \rightarrow k[U]$ be the comorphism of the morphism $U(k) \rightarrow U(k), u \mapsto x^{-1}ux$.  Let $\varphi \in \textup{Hom}_{gs/k}(\mathbb{G}_{a(r)},U)$.  We have a composite of morphisms:

\vspace{0.05in}
\begin{center}$k[U] \xrightarrow{x^{*}} k[U] \xrightarrow{\varphi^{*}} k[t]/(t^{p^r})$\end{center}
\vspace{0.05in}

We must now show that $(\varphi^{*} \circ x^{*})_Y$ and $\varphi^{*}_Y \circ x^{*}$ produce the same comorphism $k[U] \rightarrow k[t]$.  However, we note that if $b \in P(k)$, then $x^*(f)(b) = f(x^{-1}bx)$, which is to say that $x^*(f) = x.f$, and by the previous discussion this implies that $x^*(k[U]_{<p}) = k[U]_{<p}$.  Applying Lemma \ref{changes}, we have that $(\varphi^{*} \circ x^{*})_Y$ and $\varphi^{*}_Y \circ x^{*}$ must agree on $k[U]_{<p}$.  As $k[U]$ is generated as a $k$-algebra by $k[U]_{<p}$, these maps then agree on $k[U]$.

(4) First, we prove that the definition of $\text{exp}^r_U$ is not dependent on the choice $B$ (the case for $T$ is similar).  Suppose that $B^{\prime}$ is another Borel subgroup of $P$.  Then there is some $x \in P(k)$ such that $B^{\prime}(k) = x^{-1}B(k)x$.  However, as follows from the proof of (3), the root subgroups of $P$ with respect to $x^{-1}T(k)x$ will lead to a different choice of $k$-algebra generators for $k[U]$, all of which lie in the set $k[U]_{<p}$.  This is to say, the subspace $k[U]_{<p}$ is independent of the choice of $B$ or $T$, and thus so is $\text{exp}^r_U$.  We see moreover that under any automorphism, $k[U]_{<p}$ must be sent to itself, thus the arguments of (3) apply to any automorphism of $P$.

\end{proof}

\bigskip

We end this section by relating our results when $r=1$ to the to the work found in \cite[5.2]{Sei}.  To do this, we must first relate the two notions of exponentiation.  Denote by $\mathfrak{g}_a$ the Lie algebra of $\mathbb{G}_a$.  It is one dimensional, spanned by the element $(d/dt)^{(1)}$, which is defined according to

\vspace{0.05in}
\begin{center}$(d/dt)^{(1)}(t) = 1, \quad (d/dt)^{(1)}(t^i) = 0 \text{ if } i > 1$\end{center}
\vspace{0.05in}

Recall that $\mathfrak{u}$ is the Lie algebra of $U$, and suppose that $x \in \mathfrak{u}$.  With the assumption that the nilpotence class of $U$ is less than $p$, we have that $x^{[p]} = 0$.  The map sending $(d/dt)^{(1)}$ to $x$ defines a morphism of restricted Lie algebras, $\mathfrak{g}_a \rightarrow \mathfrak{u}$, which in turn produces a morphism $\varphi_x: \mathbb{G}_{a(1)} \rightarrow U_{(1)} \subseteq U$.  Let $x^n$ denote the $n$-th power of $x$ as an element in the distribution algebra $Dist(U)$ (see \cite[I.7]{J1}).  One can check that the comorphism of $\varphi_x$ is given by

\begin{equation}\label{image}
Y_{\alpha} \mapsto \sum_{n=0}^{p-1} \frac{x^n(Y_{\alpha})}{n!}t^n
\end{equation}

\noindent We have then that the comorphism of $\textup{exp}^1_{\varphi_x}: \mathbb{G}_a \rightarrow U$ sends $Y_{\alpha}$ to the same expression as in (\ref{image}), only this time as a polynomial in $k[t]$.

On the other hand, consider the map $\text{exp}:\mathfrak{u} \rightarrow U(k)$ defined in \cite[5.2]{Sei}, and let $x$ be the same element in $\mathfrak{u}$ as above.  We can define a one-parameter subgroup $\text{exp}_x: \mathbb{G}_a(k) \rightarrow U(k)$, which sends $a \in \mathbb{G}_a(k)$ to $\text{exp}(ax)$.  By \cite[5.2(i)]{Sei} this map will have tangent map $(d/dt)^{(1)} \mapsto x$, and is therefore another extension of $\varphi_x: \mathbb{G}_{a(1)} \rightarrow U$.  We claim that the morphisms $\textup{exp}^1_{\varphi_x}$ and $\text{exp}_x$ agree.  As they restrict to the same morphisms over $\mathbb{G}_{a(1)}$, and given (\ref{image}), this will be true if the comorphism of $\text{exp}_x$ sends each $Y_{\alpha}$ to a polynomial of degree $< p$ in $k[t]$.  But this must be true, as can be seen in the argument that one can obtain the exponential over $k$ by base change from the ring $\mathbb{Z}[\frac{1}{(p-1)!}]$ (see \cite[2.2, Proposition 1]{Ser} and \cite[\S 8]{M}).

\section{Extension to $G$-orbits}

If $H$ is a subgroup scheme of $G$, we can view $\textup{Hom}_{gs/k}(\mathbb{G}_{a(r)},H)$ as a subset of $\textup{Hom}_{gs/k}(\mathbb{G}_{a(r)},G)$ via the inclusion $H \hookrightarrow G$.  The left-action of conjugation by $G$ induces a left-action of $G(k)$ on $\textup{Hom}_{gs/k}(\mathbb{G}_{a(r)},G)$, and we write $g \cdot \varphi$ for the image of $g \in G(k)$ acting on $\varphi \in \textup{Hom}_{gs/k}(\mathbb{G}_{a(r)},G)$.  We also have an action of $G(k)$ on $\textup{Hom}_{gs/k}(\mathbb{G}_a,G)$, and will denote this action in the same way.

Now let $P, P^{\prime}$ be conjugate parabolic subgroups of $G$, with $P^{\prime}(k) = gP(k)g^{-1}$ for some $g \in G(k)$.  Denote by $U$ and $U^{\prime}$ the unipotent radicals of $P$ and $P^{\prime}$ respectively.  We see that for $\varphi \in \text{Hom}_{gs/k}(\mathbb{G}_{a(r)},U)$ and $\phi \in \text{Hom}_{gs/k}(\mathbb{G}_{a},U)$, the maps sending $\varphi \mapsto g \cdot \varphi$, and $\phi \mapsto g \cdot \phi$ define bijections

\vspace{0.05in}
\begin{center} $\text{Hom}_{gs/k}(\mathbb{G}_{a(r)},U) \xrightarrow{\sim} \text{Hom}_{gs/k}(\mathbb{G}_{a(r)},U^{\prime})$ \end{center}
\vspace{0.07in}

\begin{center} $\text{Hom}_{gs/k}(\mathbb{G}_{a},U)  \xrightarrow{\sim} \text{Hom}_{gs/k}(\mathbb{G}_{a},U^{\prime})$ \end{center}
\vspace{0.05in}

Suppose now that $U$ (and therefore $U^{\prime}$) has nilpotency less than $p$.  In the isomorphism between $U$ and $U^{\prime}$ (given by conjugating by $g$), it follows from the proof of Theorem \ref{exp}(4) that the isomorphism of coordinate algebras $k[U^{\prime}] \xrightarrow{\sim} k[U]$ must send $k[U^{\prime}]_{<p} \xrightarrow{\sim} k[U]_{<p}$ (as subspaces, neither being closed under multiplication).  This implies that for every $\varphi \in \text{Hom}_{gs/k}(\mathbb{G}_{a(r)},U)$, we have

\begin{equation}\label{respects}
\text{exp}^r_{U^{\prime}}(g \cdot \varphi) = g \cdot \text{exp}^r_{U}(\varphi)
\end{equation}

\bigskip
In view of this, we have almost established that the exponentiation results for $\text{Hom}_{gs/k}(\mathbb{G}_{a(r)},U)$ extend to $G \cdot \text{Hom}_{gs/k}(\mathbb{G}_{a(r)},U)$.  It remains, however, to prove that if both $\varphi$ and $g \cdot \varphi$ factor through $U$, then $\text{exp}^r_{U^{\prime}}(g \cdot \varphi) = \text{exp}^r_{U}(g \cdot \varphi)$, which by (\ref{respects}) is equivalent to proving that $\text{exp}^r_{U}(g \cdot \varphi) = g \cdot \text{exp}^r_{U}(\varphi)$.

\bigskip
\begin{thm}\label{expglobal}
Let $G$ be a reductive group, $P$ a parabolic subgroup of $G$ with unipotent radical $U$, and assume that $p$ is greater than the nilpotence class of $U$.  Then for each $r \ge 1$, the map $\text{exp}^r_U$ extends uniquely to a map

\vspace{0.05in}
\begin{center}$\textup{exp}^r: G \cdot \textup{Hom}_{gs/k}(\mathbb{G}_{a(r)},U) \rightarrow G \cdot \textup{Hom}_{gs/k}(\mathbb{G}_a,U)$\end{center}
\vspace{0.05in}

\noindent where $\textup{exp}^r(g \cdot \varphi) := g \cdot \textup{exp}^r_{U}(\varphi)$ for all $g \in G(k)$, and $\varphi \in \textup{Hom}_{gs/k}(\mathbb{G}_{a(r)},U)$.

\end{thm}

\begin{proof}
Let $r \ge 1$, and fix $B$ to be some Borel subgroup of $P$, and choose $T$ to be a maximal torus of $B$.  Again let $\Phi$ be the root system of $G$ with respect to $T$, and choose simple roots so that the root subgroups lying in $B$ come from $\Phi^+$.  Let $n \in G(k)$ be in the normalizer of $T(k)$, and set $P^{\prime}$ to be the parabolic subgroup scheme corresponding $nP(k)n^{-1}$.  Suppose that $\varphi \in \textup{Hom}_{gs/k}(\mathbb{G}_{a(r)},U)$ is such that $n \cdot \varphi$ also factors through $U$.  We have then that $n \cdot \varphi$ factors through $U \cap U^{\prime}$, the intersection of the unipotent radicals of $P$ and $P^{\prime}$.  The group $U \cap U^{\prime}$ is equal to the subgroup of $U$ generated by those $U_{\alpha}$ such that

\vspace{0.05in}
\begin{center} $\alpha \in (\Phi^+\backslash \Phi^+_J) \cap w(\Phi^+\backslash \Phi^+_J)$, \end{center}
\vspace{0.05in}

\noindent where $w$ is the element in the Weyl group of $\Phi$ represented by $n$ (this is observed for $P=B$ in \cite[I.1.3(b)]{SS}).  Let $m = |(\Phi^+\backslash \Phi^+_J) \cap w(\Phi^+\backslash \Phi^+_J)|$.  We can choose total orderings on the roots in $\Phi^+\backslash \Phi^+_J$ and $w(\Phi^+\backslash \Phi^+_J)$ and choose our root homomorphisms in such a way that

\vspace{0.05in}
\begin{center} $k[U] = k[X_1,\ldots,X_m,\ldots,X_n], \quad k[U^{\prime}] = k[Y_1,\ldots,Y_m,\ldots,Y_n]$, \end{center}
\vspace{0.05in}
\begin{center} $k[U \cap U^{\prime}] = k[Z_1,\ldots,Z_m]$, \end{center}
\vspace{0.07in}

\noindent and the inclusions of $U \cap U^{\prime}$ into $U$ and $U^{\prime}$ have comorphisms

\vspace{0.05in}
\begin{center} $i_1^{*}: k[X_1,\ldots,X_m,\ldots,X_n] \rightarrow k[Z_1,\ldots,Z_m], \quad X_i \mapsto Z_i$ (or $0$ if $i > m$), \end{center}
\vspace{0.05in}
\begin{center} $i_2^{*}: k[Y_1,\ldots,Y_m,\ldots,Y_n] \rightarrow k[Z_1,\ldots,Z_m], \quad Y_i \mapsto Z_i$ (or $0$ if $i > m$). \end{center}
\vspace{0.05in}

We see then that the lifting of the comorphism $k[U] \rightarrow k[t]/(t^{p^r})$ to $k[t]$ will factor through the projection $i_1^{*}$, and similarly for $k[U^{\prime}]$.  Thus the morphisms $\text{exp}^r_U(n \cdot \varphi)$ and $\text{exp}^r_{U^{\prime}}(n \cdot \varphi)$ both factor through $U \cap U^{\prime}$ and necessarily agree by the construction above, noting that $\text{exp}^r_U$ and $\text{exp}^r_{U^{\prime}}$ are independent of our choices of $B$ and $T$ (and hence $B^{\prime}$ and $T^{\prime}$ in $P^{\prime}$), and of our choices of root orderings for $U$ and $U^{\prime}$.  By (\ref{respects}) we have that

\vspace{0.05in}
\begin{center}$\text{exp}^r_{U}(n \cdot \varphi) = \text{exp}^r_{U^{\prime}}(n \cdot \varphi) = n \cdot \text{exp}^r_{U}(\varphi)$,\end{center}
\vspace{0.05in}

Now let $g \in G(k)$ be arbitrary.  The Bruhat decomposition of $G$ allows us to write $g = b_1nb_2$, with $n \in N(T(k)), b_i \in B(k)$.  Suppose that $\varphi$ and $g \cdot \varphi$ both factor through $U$.  Then $nb_2 \cdot \varphi$ and $b_2 \cdot \varphi$ also factor through $U$, as $b_1, b_2$ normalize $U(k)$.  By the previous arguments, along with Theorem \ref{exp}(3), we get

\begin{eqnarray*}
\text{exp}^r_U(g \cdot \varphi) & = & b_1 \cdot \text{exp}^r_U(nb_2 \cdot \varphi) = b_1n \cdot \text{exp}^r_U(b_2 \cdot \varphi)\\
& = & b_1nb_2 \cdot \text{exp}^r_U(\varphi) = g \cdot \text{exp}^r_U(\varphi)\\
\end{eqnarray*}

\bigskip
By (\ref{respects}), this proves that the map $\text{exp}^r$, given by setting

\vspace{0.05in}
\begin{center}$\text{exp}^r(g \cdot \varphi) := g \cdot \text{exp}^r_{U}(\varphi)$,\end{center}
\vspace{0.05in}

\noindent for all $g \in G(k)$, and $\varphi \in \textup{Hom}_{gs/k}(\mathbb{G}_{a(r)},U)$, is well-defined and compatible with the conjugation action of $G(k)$.
\end{proof}

\begin{remark}\label{goodone}
Suppose that $G$ is semisimple and that $p$ is good for $G$.  By the work of Nakano, Parshall, and Vella in \cite{NPV}, there exists some $P$ such that $U$ has nilpotence class less than $p$, and $\mathcal{N}_1(\mathfrak{g}) = G \cdot \mathfrak{u}$.  If we further assume that $p$ does not divide the order of the fundamental group of $G$, and that $\mathcal{N}_1(\mathfrak{g})$ is a normal variety, then it is shown in \cite{CLN} that the exponential map defined on $\mathfrak{u}$ extends to a $G$-equivariant isomorphism of varieties between $\mathcal{N}_1(\mathfrak{g})$ and $\mathcal{U}_1(G)$.  However, our theorem above establishes that even if these additional hypotheses are not satisfied, the exponential on $\mathfrak{u}$ at least extends to a $G$-equivariant map from $\mathcal{N}_1(\mathfrak{g})$ to $\mathcal{U}_1(G)$.
\end{remark}

\section{Infinitesimal One-parameter subgroups}

We now apply the preceding results to the problem of describing the infinitesimal one-parameter subgroups of a reductive group $G$.  As we have done previously, we will first establish the relevant result for the unipotent radical of a parabolic subgroup, and then proceed to dealing with $G$.  If $U$ is as in Theorem \ref{exp}, then for all $x \in \mathfrak{u}$, we continue to denote by $\text{exp}_x$ the one-parameter subgroup of $G$, given by sending $s \in \mathbb{G}_a(k)$ to $\text{exp}(sx) \in U(k) \subseteq G(k)$, and likewise for the extension of the exponential to $G \cdot \mathfrak{u}$ (the existence of which is noted in Remark \ref{goodone}).

\bigskip
The proof of this first theorem builds off of the work of G. McNinch in \cite[Theorem 9.6]{M}, who first applied the results in \cite[Proposition 5.4]{Sei} towards describing infinitesimal one-parameter subgroups.

\begin{thm}\label{description}
Let $G$ be a reductive group, and assume that $P$ is a parabolic subgroup of $G$ with unipotent radical $U$ having nilpotence class less than $p$.  Then there is a natural identification between $\textup{Hom}_{gs/k}(\mathbb{G}_{a(r)},U)$ and the variety

\vspace{0.06in}
\begin{center}$C_r(\mathfrak{u}): = \{(x_0,x_1,\ldots,x_{r-1}) \mid x_i \in \mathfrak{u}, [x_i,x_j] =0\}$.\end{center}
\vspace{0.06in}

\noindent This map sends the $r$-tuple $(x_0,x_1,\ldots,x_{r-1})$ to the morphism

\vspace{0.06in}
\begin{center}$\textup{exp}_{x_0}\textup{exp}_{x_1}^{(1)} \cdots \textup{exp}_{x_{r-1}}^{(r-1)}$,\end{center}
\vspace{0.06in}

\noindent where for $a \in \mathbb{G}_a(k)$, $\textup{exp}_{x_i}^{(i)}(a) = \textup{exp}_{x_i}(a^{p^i})$. 

\end{thm}

\begin{proof}

In \cite[5.4]{Sei}, it was shown that the set $\text{Hom}_{gs/k}(\mathbb{G}_a,U)$ can be identified with the set 

\vspace{0.06in}
\begin{center}$\{(x_0,x_1,\ldots) \mid x_i \in \mathfrak{u}, [x_i,x_j] =0, \text{ and $x_i = 0$ for $i$ large}\}$,\end{center}
\vspace{0.06in}

\noindent where $(x_0,x_1,\ldots)$ is identified with the morphism $\textup{exp}_{x_0}\textup{exp}_{x_1}^{(1)} \cdots$, the latter well-defined because of the condition that $x_i = 0$ for $i >> 0$.  As shown in \cite[Theorem 9.6]{M} (which in turn follows the work of \cite[Lemma 1.7]{SFB1}), restricting these one-parameter subgroups to $\mathbb{G}_{a(r)}$ defines an injective map $C_r(\mathfrak{u}) \rightarrow \textup{Hom}_{gs/k}(\mathbb{G}_{a(r)},U)$.  On the other hand, by Theorem \ref{exp} the restriction

\vspace{0.06in}
\begin{center}$res : \textup{Hom}_{gs/k}(\mathbb{G}_{a},U) \rightarrow \textup{Hom}_{gs/k}(\mathbb{G}_{a(r)},U)$\end{center}
\vspace{0.06in}

\noindent is surjective.  Clearly the morphisms corresponding to the elements

\vspace{0.06in}
\begin{center}$(x_0,x_1,\ldots, x_{r-1}, x_r, x_{r+1}, \ldots)$ \; and \; $(x_0,x_1,\ldots,x_{r-1},0,0,0, \ldots)$\end{center}
\vspace{0.06in}

\noindent restrict to the same element in $\textup{Hom}_{gs/k}(\mathbb{G}_{a(r)},U)$, thus the map

$$C_r(\mathfrak{u}) \rightarrow \textup{Hom}_{gs/k}(\mathbb{G}_{a(r)},U)$$

\noindent is also surjective.
\end{proof}

To extend this result to $G$, we need to know that there is a $P$ and $U$ as in the previous theorem with the property that every morphism from $\mathbb{G}_{a(r)}$ to $G$ factors through a conjugate of $U$.  For $r=1$ and $p$ good this fact follows from the work of \cite{NPV}.  For $r>1$, we need a result in \cite{LMT} due to Levy, McNinch, and Testerman, which establishes in particular that an embedding of $\mathbb{G}_{a(r)}$ must factor through a Borel subgroup of $G$.  Their result assumes that $p$ is not a \textit{torsion prime} for $G$ (see \cite[\S 2]{LMT}).  On the other hand, for the unipotent radical of a Borel subgroup to satisfy our nilpotency condition, we must assume that $p \ge h$, the Coxeter number of $G$.  This will guarantee then that $p$ is not a torsion prime, except in the case that $PSL_p$ is a simple factor of the semisimple group $[G,G]$.

\begin{thm}
Let $G$ and $P$ be as in Theorem \ref{description}, and assume further that

\vspace{0.06in}
\begin{center}$C_r(\mathcal{N}_1(\mathfrak{g})) = G \cdot C_r(\mathfrak{u}) \; \text{ and } \; \textup{Hom}_{gs/k}(\mathbb{G}_{a(r)},G) = G \cdot \textup{Hom}_{gs/k}(\mathbb{G}_{a(r)},U)$.\end{center}
\vspace{0.06in}

\noindent Then there is a natural identification between $\textup{Hom}_{gs/k}(\mathbb{G}_{a(r)},G)$ and the variety

\vspace{0.06in}
\begin{center}$C_r(\mathcal{N}_1(\mathfrak{g})): = \{(x_0,x_1,\ldots,x_{r-1}) \mid x_i \in \mathcal{N}_1(\mathfrak{g}), [x_i,x_j] =0\}$.\end{center}
\vspace{0.06in}

\noindent In particular, if $p \ge h$ and $PSL_p$ is not simple factor of $[G,G]$, then this result result holds.
\end{thm}

\begin{proof}
The hypotheses clearly imply that the map defined by

\vspace{0.06in}
\begin{center}$(x_0,x_1,\ldots,x_{r-1}) \mapsto \textup{exp}_{x_0}\textup{exp}_{x_1}^{(1)} \cdots \textup{exp}_{x_{r-1}}^{(r-1)}$\end{center}
\vspace{0.06in}

\noindent is surjective.  To see that it is injective, we can follow the proof of \cite[Theorem 9.6]{M}.  Suppose that $\textup{exp}_{x_0}\textup{exp}_{x_1}^{(1)} \cdots \textup{exp}_{x_{r-1}}^{(r-1)} = \textup{exp}_{y_0}\textup{exp}_{y_1}^{(1)} \cdots \textup{exp}_{y_{r-1}}^{(r-1)}$.  Up to conjugation, we may assume $(x_0,x_1,\cdots x_{r-1}) \in C_r(\mathfrak{u})$, so that these maps both factor through $U$.  We must have that $x_0 = y_0$ by restricting to $\mathbb{G}_{a(1)}$.  We can then multiply each morphism by $\text{exp}_{-x_0}$, since $(x_0,x_1,\cdots x_{r-1}) \in C_r(\mathfrak{u})$, thus establishing that $\textup{exp}_{x_1}^{(1)} \cdots \textup{exp}_{x_{r-1}}^{(r-1)} = \textup{exp}_{y_1}^{(1)} \cdots \textup{exp}_{y_{r-1}}^{(r-1)}$.  With $F$ denoting the Frobenius morphism on $\mathbb{G}_{a}$, this says that 

\vspace{0.05in}
\begin{center}$(\textup{exp}_{x_1} \cdots \textup{exp}_{x_{r-1}}^{(r-2)}) \circ F = (\textup{exp}_{y_1} \cdots \textup{exp}_{y_{r-1}}^{(r-1)}) \circ F$\end{center}
\vspace{0.05in}

\noindent when restricted to $\mathbb{G}_{a(r)}$.  As $F(\mathbb{G}_{a(r)}) = \mathbb{G}_{a(r-1)}$, we can  proceed by induction to see that $x_i = y_i$ for all $i$.

\bigskip
Finally, if $p \ge h$, then we may select $U$ to be the unipotent radical of a Borel subgroup of $G$, as this will have nilpotence class less than $p$.  Since $p \ge h$ implies in particular that $p$ is good for $G$, we have by \cite[Lemma 9.7]{M} that any finite set of pairwise commuting $p$-nilpotent elements in $\mathfrak{g}$ lies in the Lie algebra of some Borel subgroup.  On the other hand, if we further assume that $PSL_p$ is not a factor of $[G,G]$, then $p$ will in particular also not be a torsion prime for $G$, thus \cite[Theorem 2.2]{LMT} implies that $\textup{Hom}_{gs/k}(\mathbb{G}_{a(r)},G) = G \cdot \textup{Hom}_{gs/k}(\mathbb{G}_{a(r)},U)$.
\end{proof}

\bigskip
\noindent \textbf{Acknowledgements:} We wish to acknowledge helpful discussions and/or correspondences with Jon Carlson, Eric Friedlander, Jim Humphreys, Zongzhu Lin, George McNinch, Brian Parshall, and Julia Pevtsova.

\vspace{.2 in}
\noindent\tiny{DEPARTMENT OF MATHEMATICS \& STATISTICS, UNIVERSITY OF MELBOURNE, PARKVILLE, VIC 3010, AUSTRALIA}\\
paul.sobaje@unimelb.edu.au\\
Phone: \text{+}61 \, 401769982

\end{document}